\documentclass[12pt,reqno]{article}

\usepackage[usenames]{color}
\usepackage{amssymb}
\usepackage{graphicx}
\usepackage{amscd}

\usepackage[colorlinks=true,
linkcolor=webgreen,
filecolor=webbrown,
citecolor=webgreen]{hyperref}

\definecolor{webgreen}{rgb}{0,.5,0}
\definecolor{webbrown}{rgb}{.6,0,0}

\usepackage{color}
\usepackage{fullpage}
\usepackage{float}

\usepackage{graphics,amsmath,amssymb}
\usepackage{amsthm}
\usepackage{amsfonts}
\usepackage{latexsym}
\usepackage{epsf}

\usepackage{verbatim}
\usepackage{enumerate}

\usepackage{tikz}
\usepackage{pgfplots}
\pgfplotsset{width=10cm,height=10cm,compat=1.9}

\usepackage[blocks]{authblk}

\usepackage{algpseudocode}

\theoremstyle{plain}
\newtheorem{theorem}{Theorem}

\newtheorem{lemma}[theorem]{Lemma}
\newtheorem{proposition}[theorem]{Proposition}

\theoremstyle{definition}

\newtheorem{example}[theorem]{Example}

\theoremstyle{remark}

\begin{document}

\title{Card Games Unveiled: Exploring the Underlying Linear Algebra}
\author{Nikhil Byrapuram}
\author{Hwiseo (Irene) Choi}
\author{Adam Ge}
\author{Selena Ge}
\author{Sylvia Zia Lee}
\author{Evin Liang}
\author{Rajarshi Mandal}
\author{Aika Oki}
\author{Daniel Wu}
\author{Michael Yang}
\affil{PRIMES STEP}
\author{Tanya Khovanova}
\affil{MIT}

\maketitle

\begin{abstract}
We discuss four famous card games that can help learn linear algebra. The games are: SET, Socks, Spot it!, and EvenQuads. We describe the game in the language of vector, affine, and projective spaces. We also show how these games are connected to each other. A separate section is devoted to playing Socks with the EvenQuads deck and vice versa.
\end{abstract}

\section{The game of SET}

The game of SET is a trendy card game with a lot of mathematics hidden. A lot of cool properties of the game of SET are discussed in the wonderful book, \textit{The Joy of SET} by McMahon et al.~\cite{MGGG}.

\subsection{The cards and the game}

The game of SET is played with a deck of 81 cards. Each card is characterized by 4 attributes:
\begin{itemize}
\item Number: 1, 2, or 3 symbols.
\item Color: Green, red, or purple.
\item Shading: Empty, striped, or solid.
\item Shape: Oval, diamond, or squiggle.
\end{itemize}

A \textit{set} is formed by three cards that are either all the same or all different in each attribute. Figure~\ref{fig:setex} shows an example of a set, where for each attribute, all values are different.

\begin{figure}[ht!]
\begin{center}
\includegraphics[scale=0.25]{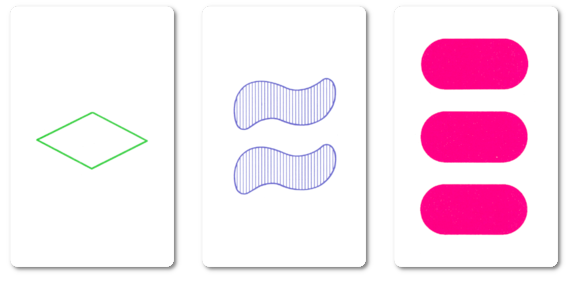}
\end{center}
\caption{An example of a set.}
\label{fig:setex}
\end{figure}

In the game of SET, the dealer lays out 12 cards. Then, without touching the cards, players look for three cards that form a set. As soon as the players find a set, they have to yell out ``set,'' and whoever says it first gets to take the three cards, as long as they are correct. If the three cards they chose do not form a set, they have to put the cards back. If players cannot find a set after a while, the dealer may lay out additional three cards. The game ends when the deck runs out, and there are no sets among the remaining cards on the table. The winner of the game is the one with the most sets.

\subsection{Interesting properties of the game of SET}

The cards in the deck are all different, and all possible cards are present. It follows that there are $3^4 = 81$ cards in a deck: for each of the four attributes, there are three possible values. 

We can check that, given two cards, a unique third card exists that completes it into a set. We can calculate this card by separately calculating the value for each attribute. If the given two cards have the same value for an attribute, then the third card must have the same value as the other two. If the two cards have different values for an attribute, the third card must have a value different from the other two. This rule uniquely determines the third card to complete a set.

We know that this third card is not a repeat card: if it were the same as one of the two cards, it would also be the same as the second card, which is a contradiction because we assumed the two starting cards are distinct.

Given three cards, it is easy to calculate the probability of these cards forming a set. Suppose we picked two cards at random. To complete them into a set, we need to choose one specific card out of 79 cards that are left over. Thus, the probability is $\frac{1}{79}$.

\subsection{Nosets and packed sets}
\label{sec:nosets}

We call a set of cards a \textit{noset} if it doesn't contain a set. It is known that the maximum size of a noset of a standard deck of set cards is 20. We can view the SET deck as a vector space $\mathbb{Z}_3^4$. We can naturally extend the noset question to the space $\mathbb{Z}_3^n$. In such a space, nosets are also called caps. Caps in $\mathbb{Z}_3^n$ have been studied in great detail. In particular, finding the maximal size of a cap has been the subject of many papers, and this problem is known as the ``Cap Set'' problem. It has been solved for $n\leq6$, and there are bounds for higher values, see a recent survey by Grochow \cite{Grochow}.


We can define a \textit{packed set} for the game of SET as a set of $n$ cards with the maximum possible number of sets. The nosets and packed sets are deeply connected. This connection can be explained through complementary sets.

We call two sets of cards \textit{complementary piles} if they are disjoint sets and their union is a full deck. Suppose we have two complementary sets: $A$ and $B$. Suppose the number of sets within $A$ is $S_A$, and the number of sets within $B$ is $S_B$. We also denote the number of cards in piles $A$ and $B$ as $|A|$ and $|B|$, respectively. 

The proof of the following proposition is done through straightforward set counting and is left to the reader.

\begin{proposition}
Suppose $A$ and $B$ are complementary piles in a deck of set cards. Then
\[S_A + S_B = \frac{|A|^2-81|A|+2160}{2} = \frac{1}{2}(2160-|A|\cdot |B|).\]
\end{proposition}

The second form of the sum shows us that it is symmetric for $A$ and $B$ (they can be swapped without changing anything). The point of the formula is that the sum of the number of sets in $A$ and $B$ depends only on their sizes. For example, the maximum number of sets in $A$ is achieved when the number of sets in $B$ is minimum, and vice versa.

This formula allows us to calculate some values. In particular, if pile $A$ is packed, then pile $B$ has the fewest number of sets possible. On the other hand, if pile $A$ is noset, then pile $B$ is packed. 

For example, suppose set $A$ has 27 cards. Then the maximum number of sets is achieved if the cards form a full deck with $3$ attributes. Indeed, for a full deck, any two cards can be completed into a set, and this is the best we can do. The number of sets in this case is $\frac{\binom{3^3}{2}}{3} = 117$. That means the minimum number of sets for a complementary pile of 54 cards is
\[\frac{1}{2}(2160-27\cdot54)-117=234.\]

As another example, we can view a pile of 61 cards to be a complement of a pile of 20. Since a pile of 20 has at minimum 0 sets, the maximum number of sets in a pile of 61 is:
\[\frac{2160-20\cdot61}{2}=\frac{940}{2}=470.\]

\subsection{Connection to linear algebra}

We can assign values 0, 1, and 2 for each attribute and consider each attribute as a coordinate. Thus we can view each card as an element in the vector field $\mathbb{Z}_3^4$. Consider the property of a set where three cards need to have all the same value or be all different. That means the values can be in one of the following sets $\{0,0,0\}$, $\{1,1,1\}$, $\{2,2,2\}$, and $\{0,1,2\}$. Equivalently, we can say that the values of one attribute in a set sum to 0 modulo 3.

Consider three cards $a$, $b$, and $c$ in the deck. We can view them as distinct vectors $\vec{a}$, $\vec{b}$, and $\vec{c}$ in $\mathbb{Z}_3^4$. Then $\{a, b, c\}$ form a set if and only if $\vec{a} + \vec{b} + \vec{c} = \vec{0}$. Consider two vectors  $\vec{a}$ and $\vec{b}$. As in Euclidean geometry, exactly one line passes through the endpoints of these vectors. What other points are on this line? Suppose $\vec{x}$ is on the same line, then $\vec{x} - \vec{a}$ should be proportional to $\vec{b} - \vec{a}$. That means $\vec{x} = \vec{a} + m(\vec{b} - \vec{a})$. But $m$ can take only three different possible values modulo 3. Plugging in $m$ equal 0, 1, or 2, we get three possible points on the line: $\vec{a}$, $\vec{b}$, and $-\vec{a} + 2\vec{b} = -\vec{a} - \vec{b}$. It follows that there are only three points on any line, and their sum is $\vec{0}$, which means they correspond to a set.

When we view the deck of SET cards as a vector space, we assign one of the cards as the origin. If we want to emphasize that all cards are equal and interchangeable, we can view a deck of SET cards as a finite affine geometry $AG(4,3)$.

\section{The Socks Game}

After the game of SET appeared, it became very popular. People tried different variations, and the following variation, named Projective SET, soon became well-liked. As far as we know, initially, people would make cards by hand. Then this game became available for sale as the game of Socks.

\subsection{The cards and the game}

The deck consists of cards featuring six colored socks, each containing a non-empty subset of 6 socks with different colors. Socks of a particular color appear in a fixed position on the card. We describe the color in the card's order: red, blue, green, pink, purple, and yellow. All cards are distinct, and there is no empty card.

A \textit{match} is a set of several cards with an even number of each color of socks. For example, a matched set of cards could have two blue, two green, two pink, and four purple socks. Figure~\ref{fig:socksmatch} shows a matched set of three cards where there are two of each color except pink.

\begin{figure}[ht!]
\begin{center}
\includegraphics[scale=0.45]{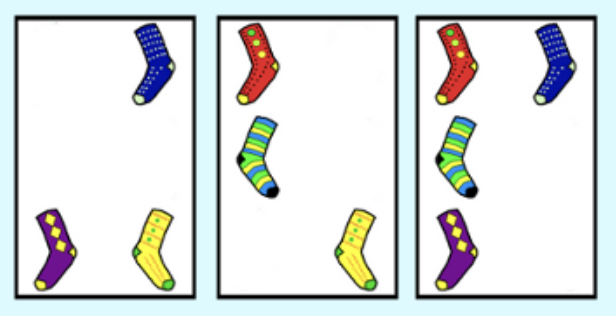}
\end{center}
\caption{A matched set of socks.}
\label{fig:socksmatch}
\end{figure}

The game is played similarly to SET. The dealer lays out 9 cards, and the players look for a match. As soon as the player finds one, they must say ``Match!''. If their chosen cards are indeed a match, the player takes the cards and gets one point per card. If incorrect, they lose a point and must put the cards back. If there is no match on the table, the dealer may lay out additional cards. The game ends when there are no more cards left on the table. The player with the most points wins.

The official rules in the game of socks are slightly different. The rules require finding a set of THREE matched cards. In this case, there are always exactly two of each kind of sock. Moreover, the official game requires 12 cards on the table. This game is too easy. So we played a more difficult version described above.

\subsection{Some interesting properties}

The deck has 63 distinct cards, meaning all possible sock combinations exist. Indeed, there are 2 ways to decide whether the color should be on the card or not for each color. Which gives 64 possible combinations of all colors. After removing an empty card, we get 63 cards left.

For any set of socks cards that is not matched, we can calculate 1 card to complete it to a matched set of socks. We need zero or one more for every type of sock to make the total even. Thus, we know the set of socks that is needed for completion. The calculated card may already be in our set. In this case, removing this card from the set makes the leftover set of cards match. Thus, for any set of socks cards that is not matched, we can either add or remove one card to make it a matched set. If the cards are matched, the calculation for the card to complete it produces an empty card. 

If we play the version of socks with matched sets of any size, after the last card is dealt, whatever is on the table has to be a matched set. This is because the whole deck has an even number of each type of sock; more precisely, there are 32 socks of each type in the whole deck. When the players remove matched sets, they remove an even number of each type of sock. Thus, the remaining cards have an even number of each type of sock, forming a matched set. The player who understands this property can yell ``Match!'' right after the last card is dealt. Such a player is guaranteed a matched set at the end as long as everyone has played correctly.

The official rules are easier for many reasons. One of them is that we only need to check triples of cards. Actually, we only need to check pairs of cards. For every pair of cards, we can calculate the 1 card that completes it into a matched set and see if this card is on the table. For a matched set of three cards and each type of sock, there could be 0 or 2 of this type of sock in the set.

However, when playing by official rules, there is no guarantee that all cards will be matched at the end of the game. Once, we played such a game and finished with six cards on the table, each with a blue sock on. We leave it to the reader to show that such a set can't be divided into two matched sets of three cards.

\subsection{Connection to linear algebra}

Every card can be described by a string of length six consisting of zeros and ones. The $i$th character corresponds to the $i$th type of socks. Zero means that the corresponding sock is not on the card, and one means it is there. All combinations of zeros and ones are possible except when all characters are zero. This is because the deck doesn't contain an empty card. Thus, the deck can be viewed as vector space $\mathbb{Z}_2^6$ with the origin removed.

Bitwise XOR is the equivalent of binary addition for each type of sock, where 1 represents the existence of a sock, and 0 represents no sock. The cards that form a matched set are exactly the vectors that sum to 0. Equivalently, these are the vectors whose bitwise XOR produces the zero vector.

In vector space $\mathbb{Z}_2^6$, each line that passes through zero has exactly one more point on it. That means we can view the cards as the space of lines passing through zero or, in other words, as a projective space.

\subsection{Compare to SET}

In the original game, only three cards are allowed as a matched set. Together with the empty card, they form a plane in the underlying vector space. In the projective space, they form a line. This is very similar to the game of SET, where three cards form a line too. No wonder this game is often referred to as a projective SET.

\subsection{Probability of a matched set of socks}

We saw that the probability of three cards in the game of SET forming a set was very easy to calculate. Not so for the game of socks. Here we calculate the probability that $n$ cards form a matched set of socks.

Let $P(n)$ be the probability of getting a matched set from $n$ cards. We know that $P(0) = P(1) = P(2) = 0$. For 3 cards, we can pick 2 cards, and only one possible card out of 61 remaining cards can complete them to a matched set. So $P(3)=1/61$. In addition, the entire deck is a matched set of socks: $P(63) = 1$. By symmetry, we have $P(n) = P(63-n)$. Indeed, the 63 cards in a Socks deck form a matched set of socks. Let us take a subset of $n$ cards. The complementary subset of $63-n$ is matched if and only if our subset is matched.

Now we find $P(n+1)$ recursively. First, we select $n$ cards. The probability that they are already a matched set is $P(n)$. If not, exactly one card will form a matched set when added to the set of $n$ cards. However, it might be a repeat of a card we already have. The probability that this is the case is $nP(n-1)$. Therefore, 
\[P(n+1)=\frac{1-P(n) -nP(n-1)}{63-n}.\]

For example, $P(4) = \frac{1-P(3) - 3P(2)}{63-3} = \frac{1-\frac{1}{61} - 0}{60} = \frac{1}{61}$.

The probabilities follow the following interesting pattern.

\begin{theorem}
We have $P(2n) = P(2n-1)$.
\end{theorem}

\begin{proof}
Consider a set $S$, chosen randomly, consisting of $2n-1$ cards and a zero card. Let $\vec{v}$ be a vector representing a card chosen randomly from the remaining $63-(2n-1)=64-2n$ cards. Next, add the vector $\vec{v}$ to the vector representing each card in $S$. The resulting vectors also correspond to a set $S'$ of cards chosen randomly from the 63 cards since the sum of two random vectors is also a random vector chosen from the whole set of cards. Moreover, the sum of vectors in $S$ is the same as the sum in $S'$, as we have added in total $2n$ copies of vector $\vec{v}$ and $2n \cdot \vec{v} = 0$ in $\mathbb{Z}_2^6$. 

Thus, the set $S'$ has matching socks if and only if the set $S$ has matching socks. Hence, the probability that $2n-1$ random cards form a matched set is the same as the probability that $2n$ random cards form a matched set.
\end{proof}

Figure~\ref{fig:graphP(n)} shows the graph of the values of $P(n)$ for odd $n < 32$. The values for even $n$ can be calculated using the property $P(2n) = P(2n-1)$, and for $n > 32$ by symmetry: $P(n) = P(63-n)$.

\begin{figure}[ht!]
\begin{center}
\includegraphics[scale=0.4]{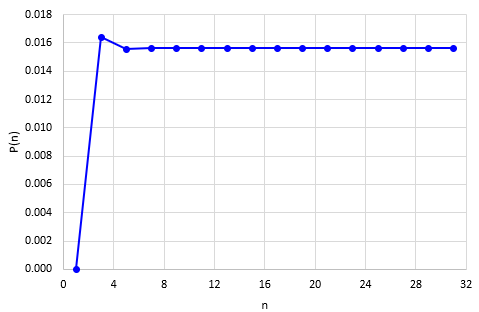}
\end{center}
\caption{Graph of $P(n)$.}
\label{fig:graphP(n)}
\end{figure}

From this graph, we see that starting from $n=5$, the values look the same. They are not. To prove this, we calculated the actual probabilities and put them into Table~\ref{table:p(n)}.

\begin{table}[ht!]
\begin{center}
\begin{tabular}{|c|c|c|c|}
\hline
$n$ & Exact value of $P(n)$       & $P(n)$ in decimal & $P(n)-\frac{1}{64}$ \\ \hline
1   & $0$                         & $0$                          & -0.01562500000                            \\ \hline
3   & 1/61                        & $0.01639344262$              & 0.00076844262                             \\ \hline
5   & 56/3599                     & $0.01555987774$              & -0.00006512226                            \\ \hline
7   & 1069/68381                  & $0.01563299747$              & 0.00000799747                             \\ \hline
9   & 11752/752191                & $0.01562369132$              & -0.00000130868                            \\ \hline
11  & 56629/3624193               & $0.01562527161$              & 0.00000027161                             \\ \hline
13  & 962672/61611281             & $0.01562493077$              & -0.00000006923                            \\ \hline
15  & 6738743/431278967           & $0.01562502119$              & 0.00000002119                             \\ \hline
17  & 18630608/1192359497         & $0.01562499233$              & -0.00000000767                            \\ \hline
19  & 980559/62755763             & $0.01562500324$              & 0.00000000324                             \\ \hline
21  & 6023432/385499687           & $0.01562499842$              & -0.00000000158                            \\ \hline
23  & 246960751/15805487167       & $0.01562500089$              & 0.00000000089                             \\ \hline
25  & 246960728/15805487167       & $0.01562499943$              & -0.00000000057                            \\ \hline
27  & 9137547511/584803025179     & $0.01562500041$              & 0.00000000041                             \\ \hline
29  & 63962829472/4093621176253   & $0.01562499966$              & -0.00000000034                            \\ \hline
31  & 703591154207/45029832938783 & $0.01562500032$              & 0.00000000032                             \\ \hline
\end{tabular}
\end{center}
\caption{The values of $P(n)$.}
\label{table:p(n)}
\end{table}
The middle values of $P(n)$ fluctuate around $0.015625$ or $\frac{1}{64}$. This mystery can be explained by the following proposition.

\begin{proposition}
If we add an empty card to the socks deck and then randomly choose $n$ cards, the probability that the chosen set has matched socks is $\frac{1}{64}$.
\end{proposition}

\begin{proof}
The bitwise XOR of $n$ random vectors in space $\mathbb{Z}_2^6$ can be any vector in $\mathbb{Z}_2^6$ with the same probability. Therefore, the probability of 0 is $\frac{1}{64}$.
\end{proof}

Now we can derive a formula for $P(n)$.

\begin{theorem}
We have
\[P(2n+1)=\frac{1}{64}-(-1)^n\frac{(2n+1)!!(61-2n)!!}{64\cdot 61!!}.\]
\end{theorem}

\begin{proof}
What is the probability that $2n+1$ cards chosen randomly, form are matched set when the empty card is allowed? Since all values of their sum are equally likely, the probability is $\frac{1}{64}$. We can calculate the same value differently. There is a $\frac{\binom{63}{2n+1}}{\binom{64}{2n+1}}$ probability the cards do not contain zero and a $\frac{\binom{63}{2n}}{\binom{64}{2n+1}}$ probability they do contain zero. If they contain zero, they are matched if the remaining cards are matched, which happens with probability $P(2n)$. If they do not contain zero, the probability they are matched is $P(2n+1)$. So
\[\frac{1}{64}=\frac{\binom{63}{2n+1}}{\binom{64}{2n+1}}P(2n+1)+\frac{\binom{63}{2n}}{\binom{64}{2n+1}}P(2n).\] Since $P(2n)=P(2n-1)$, we have $\frac{1}{64}=\frac{\binom{63}{2n+1}}{\binom{64}{2n+1}}P(2n+1)+\frac{\binom{63}{2n}}{\binom{64}{2n+1}}P(2n-1)$. Let $P(n)=\frac{1}{64}+Q(n)$. Then we have $0=\binom{63}{2n+1}Q(2n+1)+\binom{63}{2n}Q(2n-1)$. After canceling out $\frac{63!}{(2n)!(63-2n-1)!}$, we get
\[0=\frac{Q(2n+1)}{2n+1}+\frac{Q(2n-1)}{63-2n}.\]
Thus $Q(2n+1)=-\frac{2n+1}{63-2n}Q(2n-1)$. Keeping in mind that $Q(1)=-\frac{1}{64}$, it follows that
\[Q(2n+1)=-\frac{1}{64}\prod_{i=1}^n \left(-\frac{2i+1}{63-2i}\right).\] So $Q(2n+1)=(-1)^{n+1}\frac{(2n+1)!!(61-2n)!!}{64\cdot 61!!}$, and $P(2n+1)=\frac{1}{64}-(-1)^n\frac{(2n+1)!!(61-2n)!!}{64\cdot 61!!}$.
\end{proof}

\subsection{Minimal sets}

Let us call a set of matched socks cards \textit{minimal} if it doesn't contain a smaller subset of matched socks. 

\begin{lemma}
\label{lemma:minimalsets}
The minimal matched set can have 3 to 7 cards. Each number from 3 to 7 inclusive is achievable.
\end{lemma}

\begin{proof}
If we have a matched set of socks consisting of $1$ card, the card should be empty. Such a card is not in the deck. If we have a matched set of socks consisting of $2$ cards, the cards have to be identical. Thus, we do not have matched sets of socks with 1 or 2 cards.

Suppose $n$ is a number in the range 3--7. We now generate a matched set of cards of size $n$. We start with $n-1$ cards with one sock in each. These socks will necessarily be different. The last card is the card that completes the chosen cards to a matched set of socks. Namely, the last card contains all socks appearing on the chosen 1-sock cards. Such sets are minimal since 1-sock cards are all linearly independent. Thus we can have a minimal set with $3$ to $7$ cards. 

It is left to show that we can't have a minimal set with more than $7$ cards. If we have more than 7 cards, then any subset of seven cards corresponds to linearly dependent vectors in the vector space $\mathbb{Z}_2^6$. Thus, there is a subset $a_i$ of these cards in this set satisfying an equation $\sum_i a_i = 0$. This subset is a matched set.
\end{proof}

This game corresponds to a projective space over the field $\mathbb{F}_2$. It could be extended to any dimension. In the next game, the space must be a projective plane, but the underlying field can be changed.

\section{The Spot it! Game}

This game is well-known, especially among toddlers.

\subsection{The cards and the game}

In the game Spot it!, also called Dobble, the deck has round cards with 8 different symbols on each. These symbols can reappear on different cards, with 57 different symbols overall. Figure~\ref{fig:SpotIt} shows a sample of cards.

\begin{figure}[ht!]
\begin{center}
\includegraphics[scale=0.15]{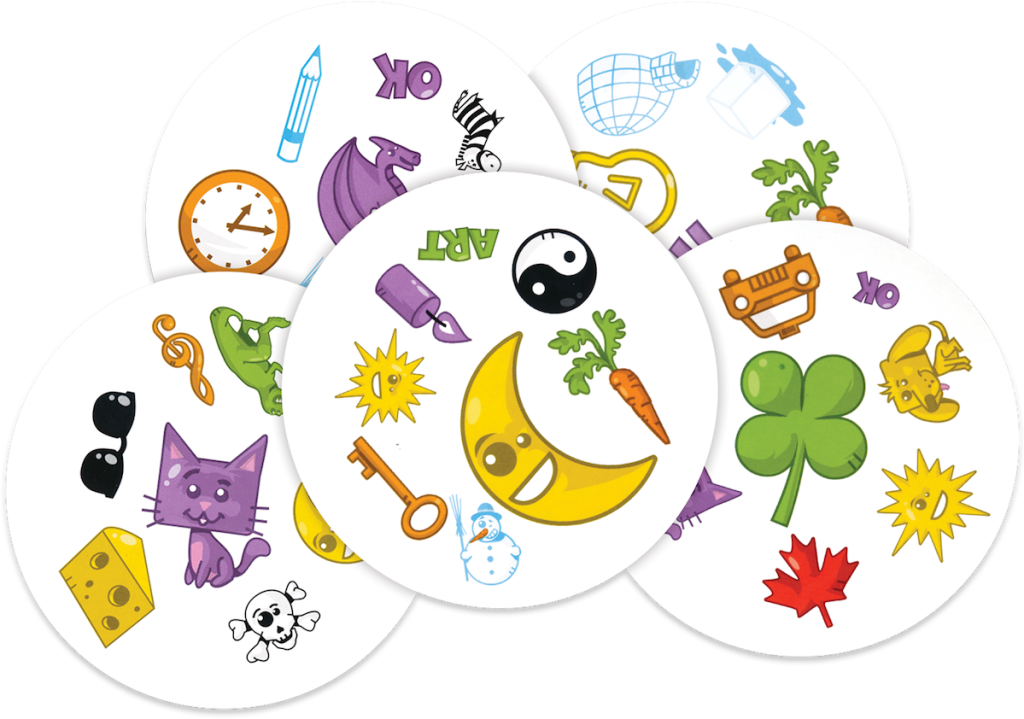}
\end{center}
\caption{A sample of Spot It! cards.}
\label{fig:SpotIt}
\end{figure}

There are many ways to play this game. We will describe the two most popular ways: the Tower and the Well.

In the Tower, one card is given face down to each player, and the remaining cards a placed in one pile in the middle of the table face up. Simultaneously, players flip their cards face up and try to spot the same symbol on the center and their own cards. When someone finds a match, they yell out the matching symbol, get the card from the top, and place it into their pile face up. The game continues with the new card in the center until the deck runs out. The player with the most cards wins.

In the Well, one card is placed face up in the center, and the remaining cards are distributed among players face down. Each player picks their top card and tries to find the matching symbol between their card and the center card. Once someone does, they yell out the matching symbol and put their card face up on the middle pile. The game continues with the new card in the center. The first person to get rid of all their cards wins.

\subsection{Interesting properties}

The interesting property that the cards have is the following. Each pair of cards share exactly one symbol. That means in both games above, each player has a chance to win a card. The winner is the fastest player. By the way, it is not clear why in the deck you buy, there are 55 cards. Mathematically the complete deck should have 57 cards, the same number as the number of symbols. The designers argued that for the Well version of the game, all but one card need to be distributed among the players, so for 55 cards total, the number 54 can be divided equally into many different players. Interestingly, this argument is not very strong. One can always set aside a few cards without damaging the property that any two cards have a symbol in common.

\subsection{Connection to linear algebra}

The game Spot It! has connections to linear algebra, like the Socks and SET. This connection explains how to build this amazing deck. But first, we need to define a projective plane.

A projective plane is similar to a regular plane, where we want every two distinct lines to intersect. We can do it by adding extra points and a line to the Euclidean plane. For every direction (i.e., a maximum set of mutually parallel lines), we add a point at infinity so that the lines in this direction will intersect at the new point. We also add a new line, which is considered incident with all the points at infinity (and no other points).

The beauty of a projective plane is that the points and the lines have similar properties: for any two points, there is exactly one line that passes through them, and for any two lines, there is exactly one point of intersection. In finite projective planes, the number of points is the same as the number of lines. It means any full deck of the game Spot It! has to have the same number of cards as the number of different symbols.

In the game Spot It!, we can view the points as symbols and, similarly, lines as cards. It follows that every two cards have one symbol in common and that for every two symbols, there is exactly one card containing both. By symmetry, we can view the points as cards and lines as symbols.

It is known that we can build a finite projective plane corresponding to a number $q$ that is a prime power. The number of points and lines in this geometry is $q^2+q+1$. Our particular deck corresponds to $q=7$. Then this means that the number of points, as well as the number of lines, in this geometry is $7^2+7+1=57$. This is the same as the number of symbols and the number of cards supposed to be in the deck. However, as mentioned before, there are only 55 cards total in the actual Spot It! deck.

Also, we know that $q+1$ points are on every line, and for every point, $q+1$ lines pass through it. So, in the Spot It! game, $q+1=8$, which means that for every card, there are 8 symbols, and likewise, every symbol appears on 8 different cards (in a potential full deck).

There is a junior version of this game corresponding to $q=5$. The deck should contain $5^2+5+1=31$ cards. This number is perfect as it matches the stated goal of having the number of cards minus 1 highly divisible. Surprisingly, this deck has 30 cards. We can't explain this discrepancy. However, we, again, have an exercise for the reader to calculate how many symbols should be on every card. More detail on the mathematics of this game can be found in the paper by Coggin and Meyer \cite{CM}.

\subsection{Compare to the game of Socks}

In the game of Socks, each of the 63 cards can be thought of as nonzero vectors in the space $Z_2^6$, which together form a projective space of order 2 and dimension 5. Meanwhile, the cards in the game of Spot It! form a projective plane of order 7 and dimension 2. Since both are played in a projective space, we can connect the two games (albeit indirectly).

To do this, we can look at simplified versions of these games, which we call Baby Spot It! and Baby Socks. Baby Spot It! corresponds to a projective plane of order 2. Hence, there are $2^2 +2 +1=7$ cards instead of 57 with 3 symbols instead of 8 on each. Baby Socks can be viewed as nonzero points in $\mathbb{Z}_2^3$. There are three types of socks instead of six. In such a deck, there will be $2^3-1=7$ cards. They form a projective plane instead of a projective space of dimension 5 for the usual game.

Thus, we get a projective plane with seven cards in both simplified games. That means we can use the Baby Socks cards as symbols in the new Baby Spot It! game. The cards in the Baby Spot It! game are lines in the projective space. Hence, each card would consist of three images of the Socks card that form a matched set. The cards can be arranged in a Fano plane. 
 In such a plane, each line (Baby Socks matched set) represents one of the Baby Spot It! cards, and each point (Baby Socks card) represents a symbol on the Baby Spot It! cards. Since in a Fano plane, there are 3 lines through each point and 3 points that make up each line, this satisfies the conditions to make a Baby Spot It! deck.

\begin{figure}[ht!]
\begin{center}
\includegraphics[scale=0.15]{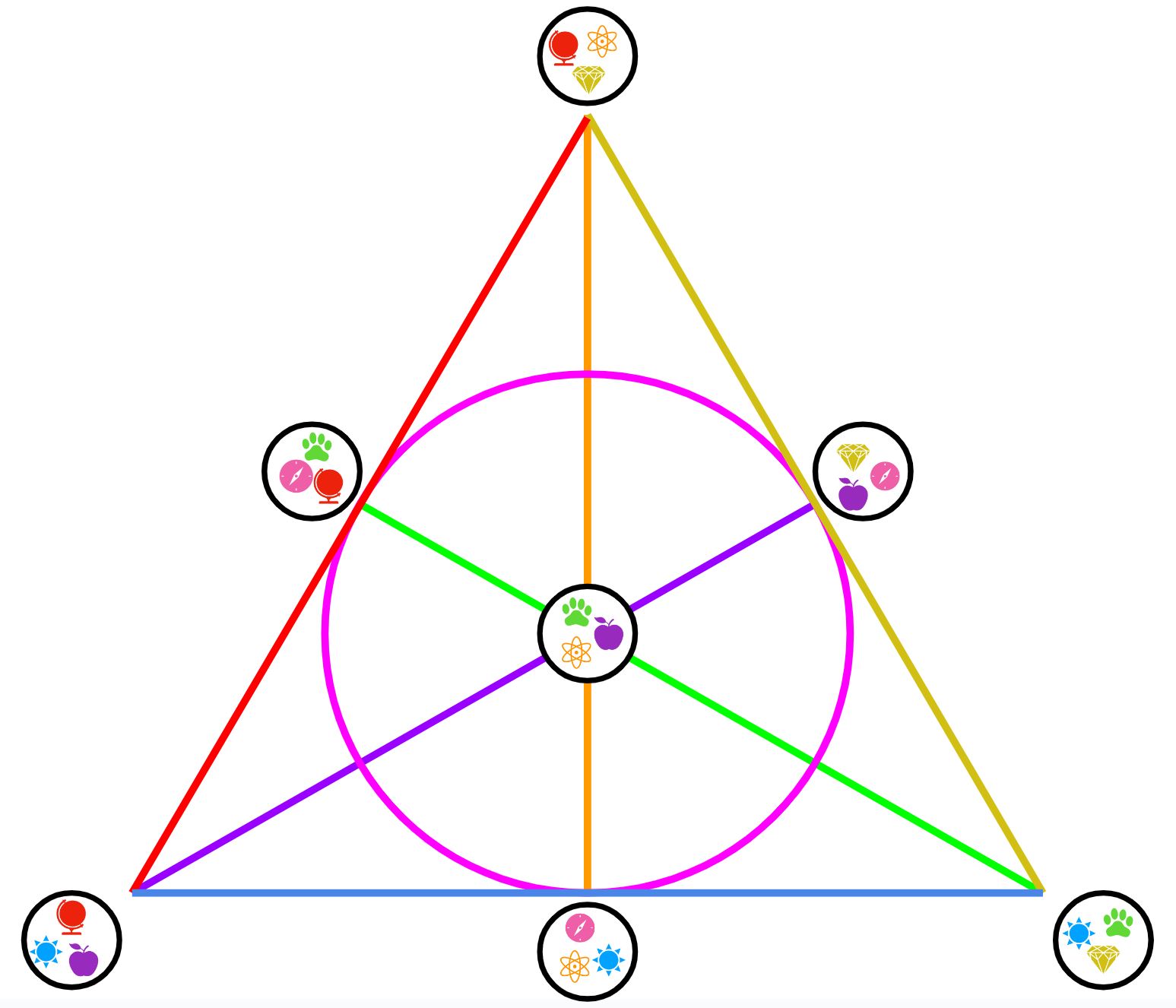}
\quad
\includegraphics[scale=0.15]{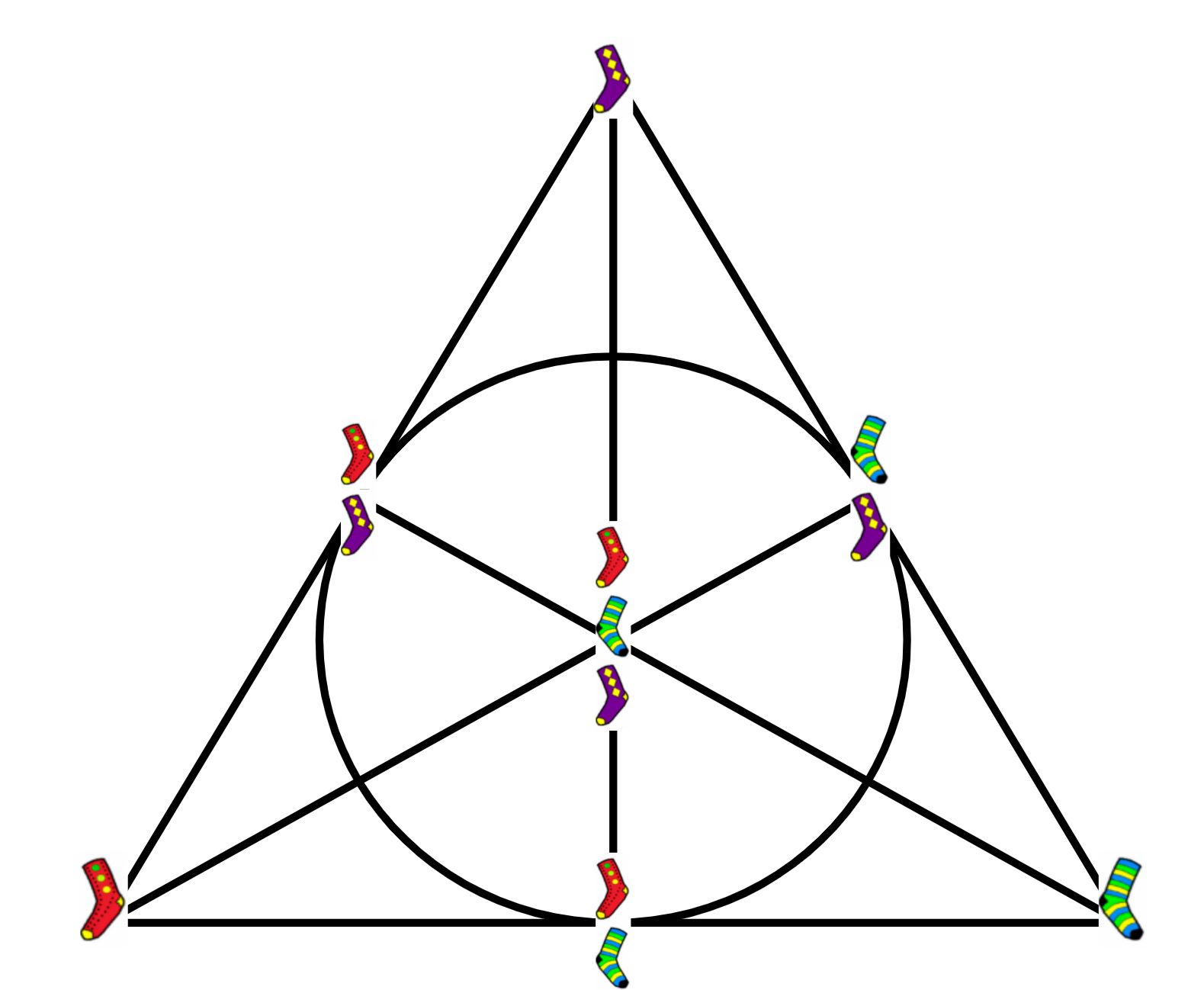}
\end{center}
\caption{Baby Spot It! and Baby Socks as the Fano plane.}
\label{fig:Baby}
\end{figure}

In some sense, the two baby games are equivalent. In Baby Socks, given two cards that can be viewed as points, we are looking for the third card that completes the line. In Baby Spot It!, we can also view two cards as points. The goal is to find the line passing through these points represented by a shared symbol on both cards. Thus, in both games, we are looking for a line in a Fano plane that passes through two given points, as seen in Figure~\ref{fig:Baby}.

Now we are ready to describe an exciting new game that also has connections to linear algebra and vector spaces.

\section{The EvenQuads game}

Many people tried to extend the game of SET to four values of each attribute. They wanted four cards to form a set and keep the property that any three cards could be uniquely completed into a set. One idea is to have four cards to form a set if an attribute's total value equals 0 modulo 4. Consider what happens with this rule when we only look at one attribute. The cards with all the same or all different values form a set, but we also have a lot of other sets. For example, we can have (0,1,1,2), but we can't have (0,2,2,1), while in the game of SET, the values are interchangeable. The sets are difficult to spot and do not have nice symmetries. A more serious problem with this approach is that the card needed to complete a set may have already been used. For example, we can take the cards 000, 123, 321. The card 000 is needed to make the set of cards equal 0 modulo 4 for each attribute, but the card has already been used.

All the issues were resolved in the game of Quads that was introduced by Rose and Perreira \cite{Rose}. It was originally thought of as a generalization of the game of SET and was initially called SuperSET. Now it is also called EvenQuads.

\subsection{The cards and the game}

The EvenQuads deck consists of $64 = 4^3$ cards with different objects. The cards have 3 attributes with 4 values each:
\begin{itemize}
\item Number: 1, 2, 3, or 4 symbols.
\item Color: Red, green, yellow, or blue.
\item Shape: Square, icosahedron, circle, or spiral.
\end{itemize}

A \textit{quad} consists of four cards so that for each attribute, the cards must be either: all the same, all different, or half and half. Figure~\ref{fig:quadex} shows an example of a quad with all attributes different.

\begin{figure}[ht!]
\begin{center}
\includegraphics[scale=0.75]{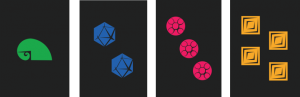}
\end{center}
\caption{An example of a quad.}
\label{fig:quadex}
\end{figure}

In EvenQuads, the dealer lays out $x$ cards from the deck facing up. Then, without touching the cards, players look for four cards that form a quad. As soon as the player finds a quad, they have to yell out ``Quad,'' and whoever says it first gets to take the four cards, as long as they are correct. If the four cards they chose do not form a quad, they have to put the cards back. If the cards do form a quad, then four new cards are drawn from the deck. If players cannot find a quad after a while, the dealer may lay out an additional card. The game ends when no more cards are left in the deck and there are no possible quads in the remaining cards. The winner of the game is the one with the most quad cards.

The number of cards $x$ might be chosen depending on the sophistication of players. The recommended number is in the range between 6 and 9 inclusive. The more cards on the table, the easier it is to find a quad.

The symbols on the deck are not random. They are logos of famous mathematical organizations! The Mathematical Association of America (MAA)'s logo is the icosahedron. The National Association of Mathematicians (NAM)'s logo is the square. The Association for Women in Mathematics (AWM)'s logo is the spiral. The Women and Mathematics Education (WME)'s logo is the circle. The logos are presented in Figure~\ref{fig:logos}.

\begin{figure}[ht!]
\begin{center}
\includegraphics[height=1.7cm]{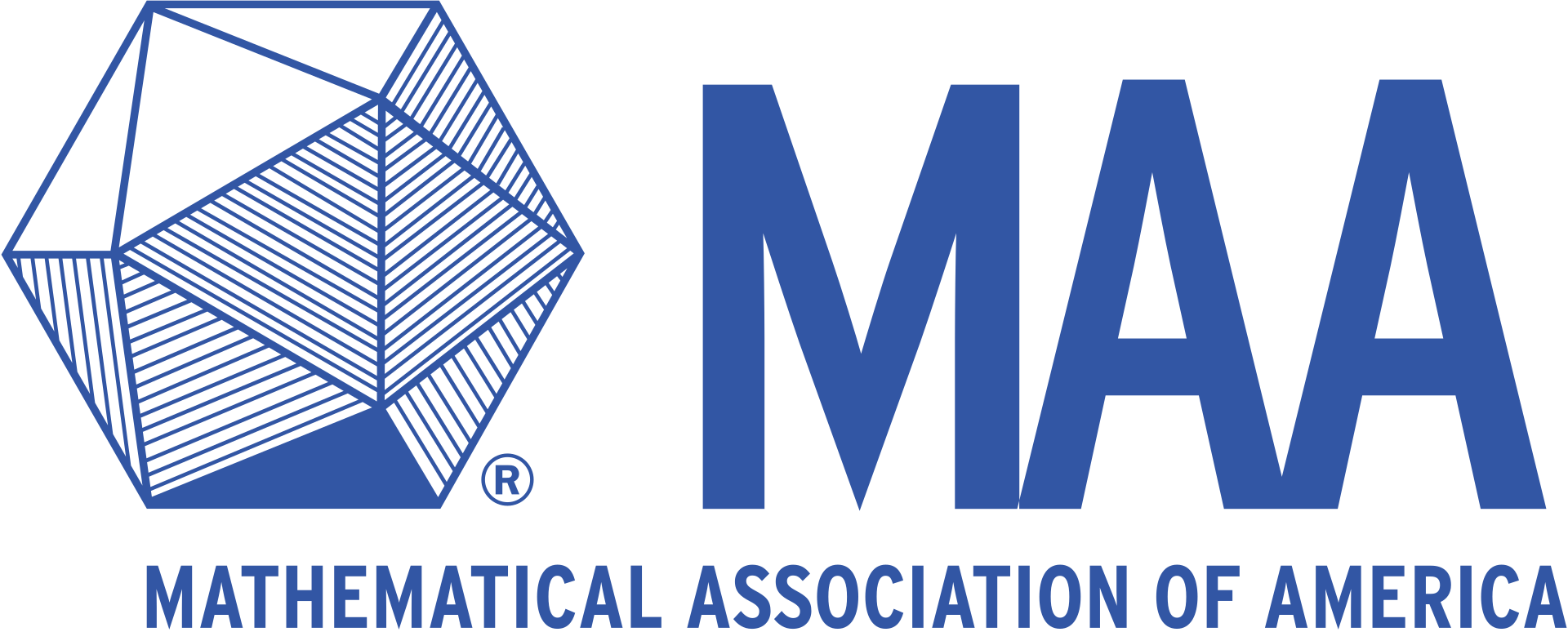}
\quad
\includegraphics[height=1.7cm]{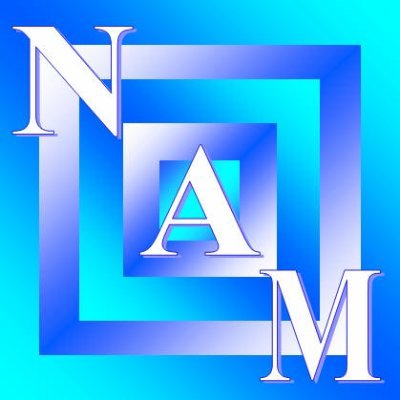}
\quad
\includegraphics[height=1.7cm]{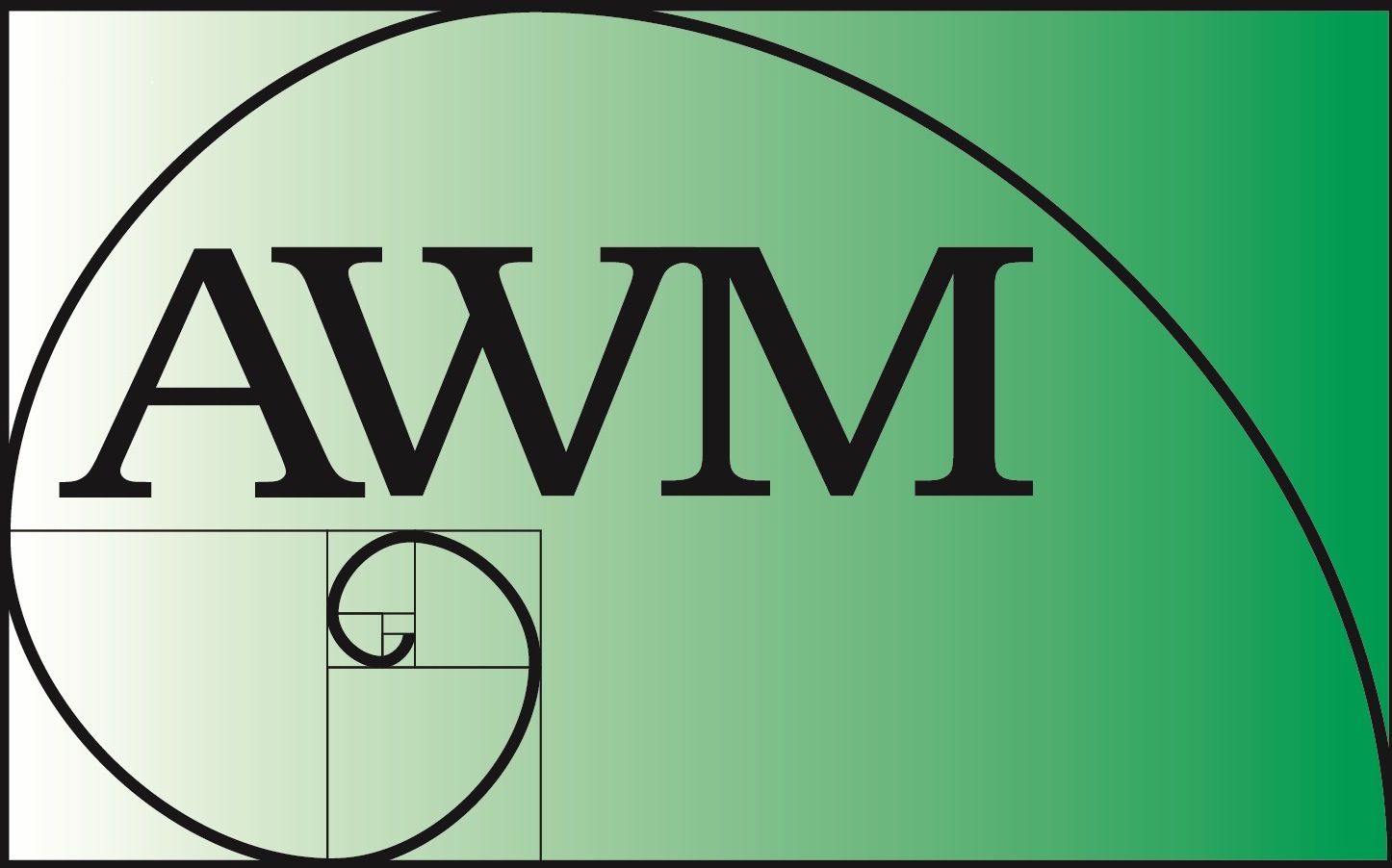}
\quad
\includegraphics[height=1.7cm]{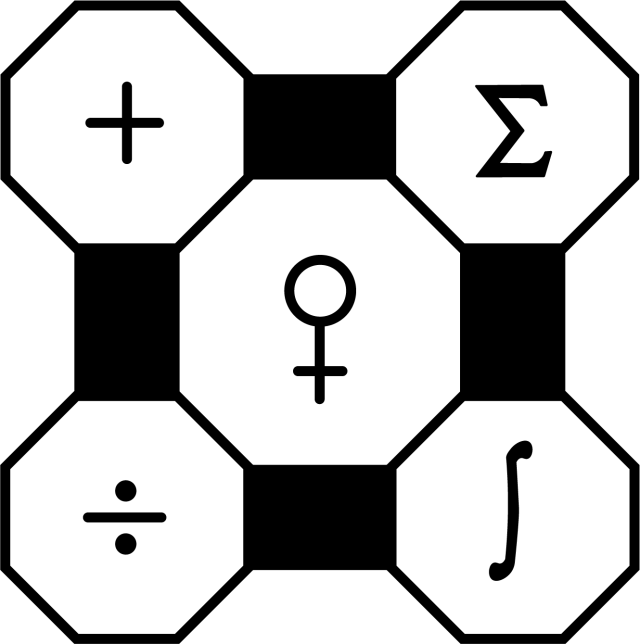}
\end{center}
\caption{Logos that inspired EvenQuads symbols.}
\label{fig:logos}
\end{figure}

\subsection{Interesting properties of the game of EvenQuads}

Similar to the game of SET, any three cards can be completed to a quad. We can prove this by looking separately at each attribute. The value of this attribute for three cards can be either all different, all the same, or two-plus-one. In any case, we know the value of the fourth card.

It follows that the probability of four cards making a quad is $\frac{1}{61}$. And the total possible number of quads is $\frac{\binom{64}{4}}{61} = 10416$. We can extend this result to a deck with $m$ different attributes. It is easy to see that the proposition still holds, so the probability of four cards making a quad is $\frac{1}{4^m - 3}$, and the total number of quads is $\frac{\binom{4^m}{4}}{4^m - 3}$.

Remember the notion of a noset in the game of SET we described in Section~\ref{sec:nosets}. Nosets, also called caps, are sets of cards that do not contain a set. In the game EvenQuads, we call a \textit{noquad} or \textit{2-cap} a set of cards without quads. It is known that the maximum size of a noset for the standard deck of the game of SET is 20. What is the maximum number of quad cards in a noquad? This and many other questions were answered in Crager et al.~\cite{CragerEtAl}. In particular, 10 cards always contain a quad. Thus, the maximum size of noquad is 9 cards. By the way, the probability that 9 cards do not contain a quad is about 0.036.

There are a lot of similarities between the games of SET and EvenQuads. In the first game, we have four features with three attributes each, and vice versa, in the second case. We can build an intersection deck in a sense. We can take all red cards from the SET deck. We call this deck a Q-SET to emphasize the connection to EvenQuads. The Q-SET deck has 27 cards with three different features and three attributes each. We can build an S-Quads deck by removing all red cards, all icosahedrons, and all cards with four shapes. The S-Quads deck has 27 cards with three different features and three attributes each. The Q-SET and S-Quads deck are isomorphic to each other. So we played the EvenQuads game with the Q-SET deck and the SET game with the S-Quads deck. It somewhat messed with our brains breaking the patterns we created for playing SET and EvenQuads. But it was fun. Unfortunately, the games ended too fast.

\subsection{Connection to linear algebra}

It might be tempting to look at the set of quad cards as a vector space $\mathbb{Z}_4^3$. But in this case, how do we describe the rule of a quad? If we consider numbers modulo 4, then in a quad, they sum to an even number modulo 4. But this condition is insufficient, as numbers 0, 0, 1, and 3 sum to zero modulo 4, but they do not satisfy the quads rule.

But we can sum the numbers differently. For example, we can use the bitwise XOR. Bitwise XOR for several bits returns 0 if the number of 1s is even and 1 if the number of 1s is odd. When we take the bitwise XOR sum of multiple binary strings, each bit in the resulting string is the bitwise XOR of the bits at the same position in the other strings. It is the same as the addition in base 2 but without carrying. We can convert the values 0, 1, 2, and 3 to binary, which are 00, 01, 10, and 11, respectively. If the four values for an attribute are all different, the bitwise XOR sum is 0. If two of the same value exist, then the bitwise XOR sum is 0 if and only if the remaining two values are equal, making the attribute either half and half or the same. So if the bitwise XOR sum of the values for an attribute in a set of four cards equals 0, it is a quad.

This is equivalent to saying that each attribute takes values in $\mathbb{Z}_2^2$. Thus, we can view our cards as vectors in $\mathbb{Z}_2^6$. Then four vectors $\vec{a}$, $\vec{b}$, $\vec{c}$, and $\vec{d}$ form a quad if and only if (see \cite{CragerEtAl})
\[\vec{a} + \vec{b} + \vec{c} + \vec{d} = \vec{0}.\]

If you remember, the ends of three vectors corresponding to a set formed a line in their vector space. What can we say about the ends of vectors corresponding to a quad? Consider four vectors $\vec{a}$, $\vec{b}$, $\vec{c}$, and $\vec{d}$. By a translation, we can assume that $\vec{a}$ is the origin. Then, from the above equation we get $\vec{d} = - \vec{b} - \vec{c} = \vec{b} + \vec{c}.$ This means that vector $\vec{d}$ is in the same plane as the origin and vectors $\vec{b}$ and $\vec{c}$. Thus, four cards form a quad if and only if their endpoints belong to the same plane.

There is a natural approach to symmetries of quads using linear algebra. Our space is an affine space. (which is the same as a vector space, but without choosing the origin). The symmetries of the space are affine transformations. Affine transformations are parallel translations and linear transformations.

\subsection{Visualization}

We can visualize the elements of $Z^6_2$ in a plane by arranging the 64 cards in an 8-by-8 grid, as suggested in \cite{CragerEtAl}. We divide the grid into four 4-by-4 squares, and each 4-by-4 square determines the first two coordinates. The top left square is labeled 00, the top right is 01, the bottom left is 10, and the bottom right is 11. Then the 4-by-4 squares are each divided into four 2-by-2 squares. Each of these 2-by-2 squares determines the next two coordinates and is labeled similarly. The 2-by-2 squares are each divided into 4 unit squares, determining the last two coordinates. Figure~\ref{fig:coordinates} shows the labeling for each pair of digits. 
\begin{figure}[ht!]
    \centering
    \includegraphics[width=10cm]{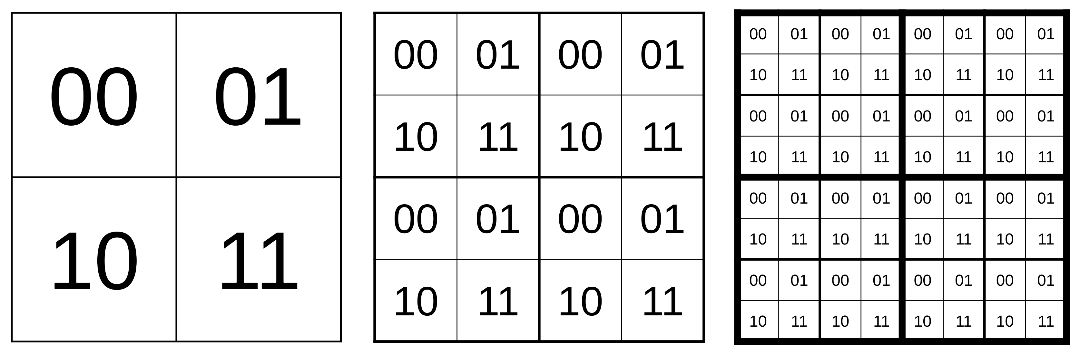}
    \caption{Coordinate labeling}
    \label{fig:coordinates}
\end{figure}

Figure~\ref{fig:coordinates2} shows an example where the green square's coordinates are 011100.
\begin{figure}[ht!]
    \centering
    \includegraphics[width=3.5cm]{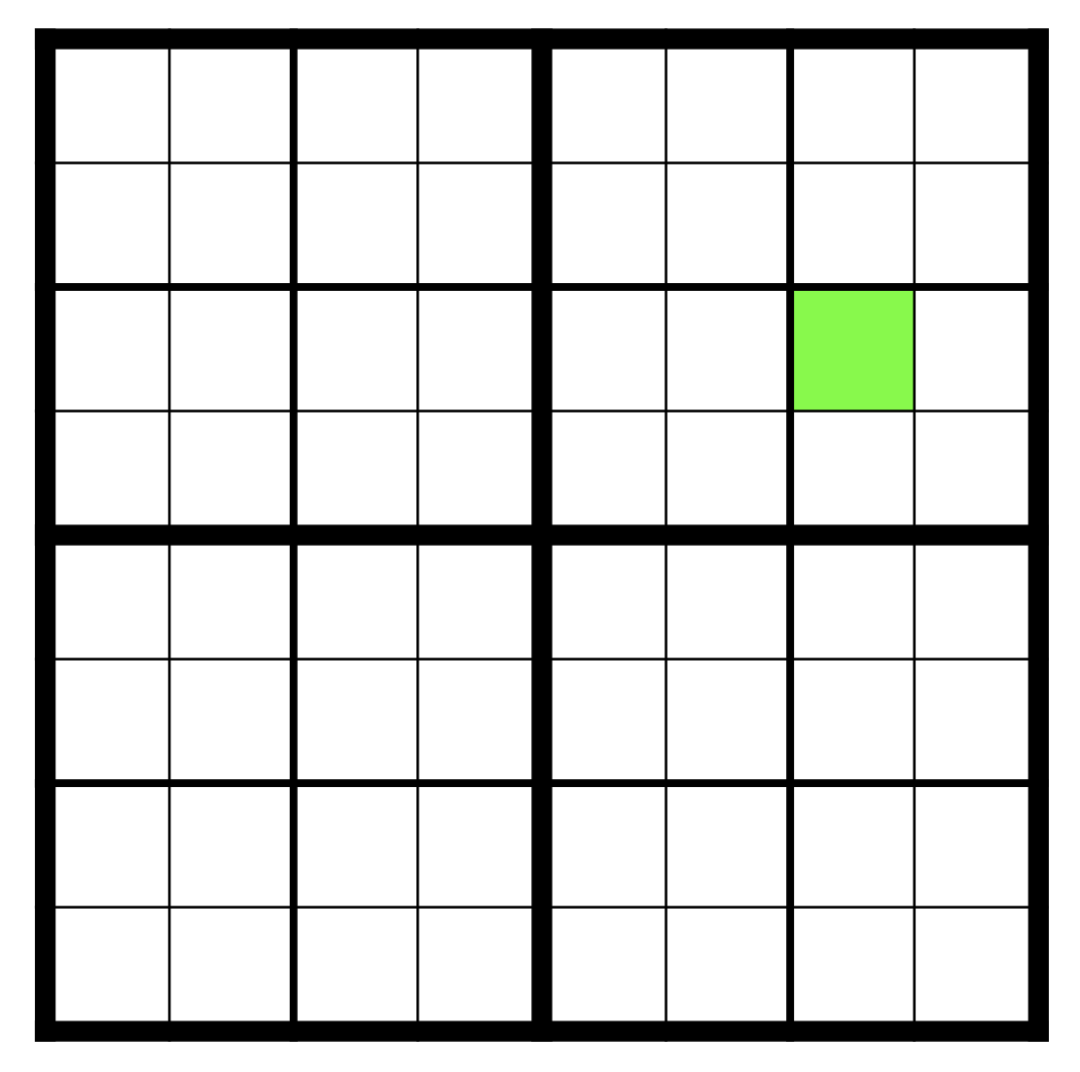}
    \caption{Example: the green square's coordinates are 011100.}
    \label{fig:coordinates2}
\end{figure}

Let's denote S to mean that an attribute is the same across all four cards, H to mean that two sets of two cards have the same attribute within each set but different overall, and D to mean that an attribute is different for all cards. Then, the possible ways to form a quad are SSD, SHH, SHD, SDD, HHH, HHD, and DDD, up to the reordering of the attributes. We excluded cases of SSS and SSH because they are impossible. SSS means that all 4 cards have the same value for each attribute. It follows that the four cards are identical. SSH means that for two of the attributes, all four cards have the same value. Then the two cards that have the same third attribute must be identical.

We give an example of how to visualize a quad of type SHD. As the first attribute is S, we have all four of its cells in one 4-by-4 square with two cells in one 2-by-2 square and the remaining two cells in another 2-by-2 square. If we look at these two 2-by-2 squares as if we were to overlap them, they would have to make a completely filled 2-by-2 square, meaning that the four squares each occupy one of 00, 01, 10, and 11. Figure~\ref{fig:SHD} shows two such examples.
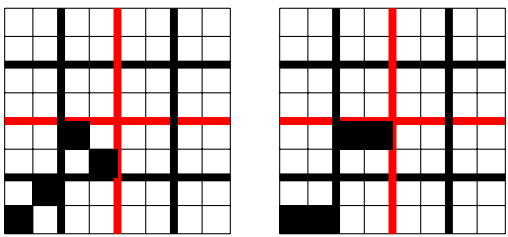
\begin{figure}[ht!]
 \begin{center}
     \begin{tikzpicture}[scale=0.75]
    \draw[step=0.5cm,color=black] (0,0) grid (4.0,4.0);
    \draw[line width=3.0,color=red] (0,2) - - (4,2);
    \draw[line width=3.0] (0,1) - - (4,1);
    \draw[line width=3.0] (0,3) - - (4,3);
    \draw[line width=3.0] (1,0) - - (1,4);
    \draw[line width=3.0,color=red] (2,0) - - (2,4);
    \draw[line width=3.0] (3,0) - - (3,4);
    \fill[black] (0,0) rectangle (0.5,0.5);
    \fill[black] (0.5,0.5) rectangle (1,1);
    \fill[black] (1,2) rectangle (1.5,1.5);
    \fill[black] (1.5,1.5) rectangle (2,1);
        \end{tikzpicture}
\quad
\begin{tikzpicture}[scale=0.75]
    \draw[step=0.5cm,color=black] (0,0) grid (4.0,4.0);
    \draw[line width=3.0,color=red] (0,2) - - (4,2);
    \draw[line width=3.0] (0,1) - - (4,1);
    \draw[line width=3.0] (0,3) - - (4,3);
    \draw[line width=3.0] (1,0) - - (1,4);
    \draw[line width=3.0,color=red] (2,0) - - (2,4);
    \draw[line width=3.0] (3,0) - - (3,4);
\fill[black] (0,0) rectangle (1,0.5);
\fill[black] (1,1.5) rectangle (2,2);
    \end{tikzpicture}
 \end{center}
     \caption{Visualizing quads of type SHD.}
    \label{fig:SHD}
\end{figure}
 
We showed how to play the game of SET with a subdeck of EvenQuads. But the most interesting connection is between the games of EvenQuads and Socks. We will discuss this in the next chapter.

\section{Relation between Socks and EvenQuads}

Here is a brief review of some necessary information:
\begin{itemize}
    \item The cards of EvenQuads and Socks can both be thought of as belonging to $\mathbb{Z}_2^6$. In Socks, each dimension represents whether a sock is on the card or not, and in EvenQuads, every two dimensions represent a value of an attribute.
    \item There are 63 cards in Socks and 64 in EvenQuads. Socks deck has 1 fewer card than $2^6$ because it misses the origin in $\mathbb{Z}_2^6$.
    \item In both games, a set can only be made if the bitwise XOR sum of the coordinates of each card is 0. However, Socks allows for any number of cards, while the game of EvenQuads restricts each set to exactly 4 cards.
\end{itemize}

\subsection{Correspondence}

Given any card in Socks, we correspond it to one card in EvenQuads, where each row in a Socks card corresponds to one attribute in an EvenQuads card. The way to calculate this correspondence is to turn the sequence of socks in each row into a binary number. Depending on the value of this number, we can determine the corresponding EvenQuads card.

We represent a Socks card as a $6$-digit binary number. If the first digit is $0$, sock $a$ is not on the card. If the first digit is $1$, sock $a$ is on the card. If the second digit is $0$, sock $b$ is not on the card. If the second digit is $2$, sock $b$ is on the card. Similarly, the presence of socks $c$, $d$, $e$, and $f$ is determined by the third, fourth, fifth, and sixth digits; see Figure~\ref{fig:sockslabeling}.

\begin{figure}[ht!]
\begin{center}
\includegraphics[scale=0.55]{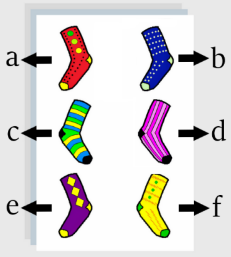}
\end{center}
\caption{Socks labeling.}
\label{fig:sockslabeling}
\end{figure}
To convert a 6-digit binary number into an EvenQuads card, we correspond the first two digits to the number, the next to digits to the color, and the last two digits to the shape. We describe the particular values in Table~\ref{table:quadslabeling}. 


\begin{table}[ht!]
\begin{center}
\begin{tabular}{|c|c|c|c|}
\hline
     &First two digits & Middle two digits & Last two digits \\
\hline
     00 & 1 Symbol & Red Symbol & Square  \\
\hline
     01 & 2 Symbols & Green Symbol & Icosahedron  \\
\hline
     10 & 3 Symbols & Yellow Symbol & Circle  \\
\hline
     11 & 4 Symbols & Blue Symbol &  Spiral \\
\hline
\end{tabular}
\end{center}
\caption{Quad cards labeling.}
\label{table:quadslabeling}
\end{table}

We should keep in mind that the Socks deck has 63 cards, while the EvenQuads deck has 64 cards. We can assign to an empty Socks card the value zero. Then, the corresponding Quads card is One Red Square.

Consider an example. The card $011010$ can be visualized as an EvanQuads or Sock card. The Quads card is 2 Yellow Circles. The Socks card has blue, green, and purple socks. see Figure~\ref{fig:qscorr}.

\begin{figure}[ht!]
\begin{center}
\includegraphics[scale=0.4]{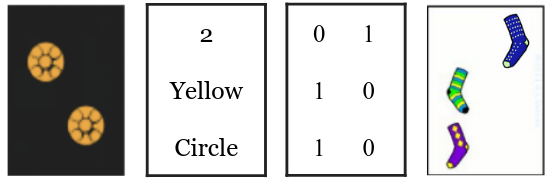}
\end{center}
\caption{An example of Quads-Socks correspondence.}
\label{fig:qscorr}
\end{figure}
We can now play Quads using Socks cards and Socks using Quads cards!

\subsection{How to play EvenQuads with the Socks deck}

We should add an empty card to the Socks deck to match the number of cards required. After that, a quad corresponds to four cards with matched socks. If one of the cards is empty, we are looking for matched socks with three cards. Thus, playing EvenQuads with a Socks deck is like playing Socks while looking for matched sets of sizes 3 and 4.

\subsection{How to play Socks with an EvenQuads deck}

We must set aside one of the sixty-four cards to be the origin (which is missing from the Socks deck). Let us call this card the \textit{origin card}. Thus, the card will have a value of 00 for each attribute (number, color, shape). For making sets of socks, we must have the used EvenQuad cards sum up to 00 in bitwise XOR for each attribute.

We can consider one attribute in quads. It corresponds to two particular socks, or equivalently, two coordinates in $\mathbb{Z}_2^6$.

Consider the following groupings of cards that, in a particular EvenQuads attribute, bitwise XOR to zero.
\begin{enumerate}
    \item \textbf{Single-zero.} A single card with the same value as the origin card, that is, with the value zero.
    \item \textbf{Double-same.} Two cards with the same value.
    \item \textbf{Triple-diff-nonzero.} A set of three distinct nonzero values of an attribute (all the colors, numbers, or shapes not in the origin card, as $10\oplus 01 \oplus 11=00$).
\end{enumerate}

\begin{proposition}
A set of EvenQuads cards corresponds to a matched set of socks set if and only if, for each attribute, we can partition the set of cards into groups that fall into one of the three categories: single-zero, double-same, and triple-diff-nonzero.
\end{proposition}

\begin{proof}
For each group, the bitwise XOR of the values is zero for the given attribute. Therefore, the total bitwise XOR is also zero. Thus, the cards that can be partitioned into these groups for each attribute form a matched set of socks.

On the other hand, we can use the greedy algorithm to prove that a group of EvenQuads cards can be divided into the groups above. Suppose we have a set of cards that bitwise XOR to zero in every attribute. First, let us put aside all cards with a value of 0 in the given attribute. This is a single-zero pile. Now, we are only left with nonzero cards. Next, we can put aside all cards that are in pairs, aka double-same cards. If there are cards left, they have to be nonzero and distinct. Moreover, they have to bitwise XOR to zero. The only way this can happen is if they are all different. That is, they belong to a triple-diff-nonzero group.
\end{proof}

\begin{example}
\label{ex:7cards7quads}
Figure~\ref{fig:7cards7quads} shows seven cards that form a matched set. This set of matched socks is 3 double-sames and one single zero for color, shape, and number. These cards contain seven quads, which we encourage our readers to find.
\begin{figure}[ht!]
\begin{center}
\includegraphics[scale=0.5]{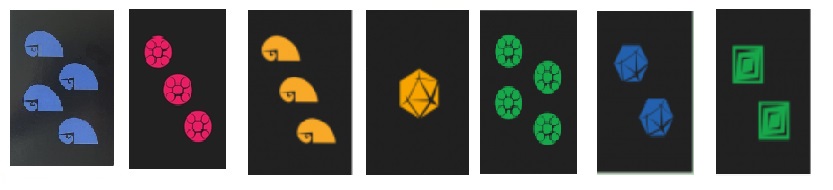}
\end{center}
\caption{Seven cards forming a matched set.}
\label{fig:7cards7quads}
\end{figure}
\end{example}

We now show how to recognize matched sets of socks in the EvenQuads deck. By Lemma~\ref{lemma:minimalsets}, we only need to study 3 to 7 cards. Recall that we assume the origin card is removed from the deck. Notice that two single-zeros are the same as one double-same, and two triple-diff-nonzero are the same as three double-sames. Thus, we can assume that we have no more than one single-zero and no more than one triple-diff-nonzero. Table~\ref{table:summaryquads2socks} shows all possibilities for a different number of cards. The number in the left column is the number of cards. The other columns represent possible combinations.

\begin{table}[ht!]
\begin{center}
\begin{tabular}{|c|c|c|c|}
\hline
Number of Cards&Single-Zero&Double-Same&Triple-Diff-Nonzero \\
\hline 
3&1&1&0 \\
\hline
3&0&0&1 \\
\hline 
4&1&0&1 \\
\hline 
4&0&2&0 \\
\hline 
5&1&2&0 \\
\hline 
5&0&1&1 \\
\hline 
6&1&1&1 \\
\hline 
6&0&3&0 \\
\hline 
7&0&2&1 \\
\hline
7&1&3&0 \\
\hline 
\end{tabular} 
\end{center}
\caption{Distributions for matched sets of EvenQuads cards.}
\label{table:summaryquads2socks}
\end{table}

For three or four cards, the description is especially concise. Three EvenQuads cards correspond to a set of matched socks if and only if adding the origin card turns this set into a quad. Four cards correspond to a set of matched socks if and only if they form a quad. If we have more cards, then to describe when they form a matched set, we need to look separately at each attribute:

\begin{itemize}
\item Five cards correspond to a set of matched socks if and only if adding the origin card turns this set into a quad and double-same for every attribute.
\item Six cards correspond to a set of matched socks if and only if they form a quad and double-same for every attribute.
\item Seven cards correspond to a set of matched socks if, for each attribute, we have two quads after adding the origin card to the set.
\end{itemize}

Playing Socks with the EvenQuads deck is difficult, as we have to choose a random card to be the origin card and keep track of it. Also, the matched set might include up to 7 cards, making it even more difficult. Also, the description of how to find matched sets is complicated. 

However, we played the original Socks game using the EvenQuads deck, which was relatively easy. In the original game, we had to find three cards with matched sets of socks. That means we were looking for three cards that can be completed into a quad by adding the origin card. We quite enjoyed this game.

Playing Socks with the EvenQuads deck wasn't only fun; it was also very educational. It helped us understand the difference between vector spaces and affine spaces. In an affine space, all vectors are equivalent, and there is no origin. That means all the cards are equivalent.

Studying all these games was educational too. It prepared us for linear algebra.

\section{Acknowledgments}

We are grateful to MIT PRIMES STEP and its director, Slava Gerovitch, for giving us the opportunity to conduct this research.

\end{document}